\documentclass[reqno]{amsart}

\usepackage{amsmath,amssymb,amsthm}

\usepackage{tikz}

\usepackage{caption}
\usepackage{graphicx}
\usepackage{graphics}

\newtheorem*{thm}{Theorem}
\newtheorem{lemma}{Lemma}
\newtheorem*{proposition}{Proposition}

\begin{document}

\title[]{A Stability Version of the \\Jones Opaque Set Inequality}

\author[]{Stefan Steinerberger}
\address{Department of Mathematics, University of Washington, Seattle}
\email{steinerb@uw.edu}

\begin{abstract}  Let $\Omega \subset \mathbb{R}^2$ be a bounded, convex set. A set $\mathcal{O} \subset \mathbb{R}^2$ is an opaque set (for $\Omega$) if every line that intersects $\Omega$ also intersects $\mathcal{O}$. What is the minimal possible length $L$ of an opaque set? The best lower bound $L \geq |\partial \Omega|/2$ is due to Jones (1962). It has been remarkably difficult to improve this bound, even in special cases where it is presumably very far from optimal. We prove a stability version: if $L - |\partial \Omega|/2$ is small, then any corresponding opaque set $\mathcal{O}$ has to be made up of curves whose tangents behave very much like the tangents of the boundary $\partial \Omega$ in a precise sense.
\end{abstract}

\maketitle

\section{Introduction and Result}

\subsection{Introduction}

Let $\Omega \subset \mathbb{R}^2$ be a bounded convex domain. An \textit{opaque set} $\mathcal{O} \subset \mathbb{R}^2$ is a set such that every line that intersects $\Omega$ also intersects $\mathcal{O}$. How short, depending on $\Omega$, can such an opaque set be? First introduced by Mazurkiewicz \cite{maz} in 1916, these sets remain poorly understood; the problem is open for all $\Omega$.

\begin{center}
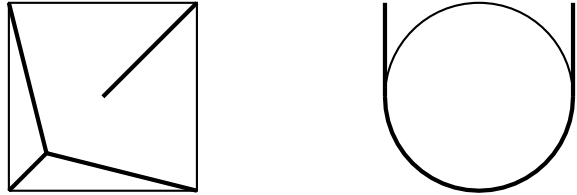
\begin{figure}[h!]
\begin{tikzpicture}[scale=2.5]
\draw [thick] (4-1.5, 0.5) circle (0.5cm);
   \draw [ultra thick,domain=180:360] plot ({4-1.5 + 0.5*cos(\x)}, {0.5 + 0.5*sin(\x)});
   \draw [ultra thick] (4-1, 0.5) -- (4-1,1);
      \draw [ultra thick] (4-2, 0.5) -- (4-2,1);
\draw [thick] (0,0) -- (1,0) -- (1,1) -- (0,1) -- (0,0);
\draw [ultra thick] (0, 1) -- (0.2, 0.2) -- (1, 0);
\draw [ultra thick] (0, 0) -- (0.2, 0.2);
\draw [ultra thick] (0.5, 0.5) -- (1,1);
\end{tikzpicture}
\caption{Conjectured shortest opaque set for $[0,1]^2$ with length $\sqrt{2} + \sqrt{3/2} \sim 2.63$ (left) and opaque set for the unit disk (right).}
\end{figure}
\end{center}
\vspace{-20pt}
In the case of $\Omega = [0,1]^2$, the trivial lower bound $\sqrt{2}$ was improved by Bagemihl \cite{bag} to $\pi/2$ and then by Jones \cite{jones} in his PhD thesis to $2$. Jones' relatively simple argument generalizes nicely and gives a universal estimate. 

\begin{thm}[Jones, 1962] If $\Omega \subset \mathbb{R}^2$ is convex, then any opaque set has length
$$ L \geq \frac{|\partial \Omega|}{2}.$$
\end{thm}

This bound appears to be difficult to improve. For $\Omega = [0,1]^2$, the best lower bound is $L \geq 2.0002$ due to  Kawamura, Moriyama, Otachi \& P\'ach \cite{kawa}. For the equilateral triangle with side length 1 the Jones bound guarantees $L \geq 3/2$ and the best known result is $L \geq 3/2 + 5 \cdot 10^{-13}$ due to Izumi \cite{iz}.
 The case of the unit disk is particularly tricky, it is sometimes known as the \textit{beam detection constant} problem \cite{finch}: it is listed as problem A30 in \cite{croft} and discussed in a 1995 issue of \textit{Scientific American} \cite{stewart}, see also \cite{joris, schneider}. The best known bounds are  $\pi < L \leq 4.8$, no significant improvement over Jones' bound is known.

\begin{center}
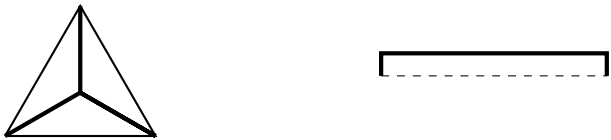
\begin{figure}[h!]
\begin{tikzpicture}[scale=1]
\draw [dashed] (0,0.8) -- (3,0.8) -- (3, 1.1) -- (0, 1.1) -- (0,0.8);
\draw [ultra thick] (3,0.8) -- (3,1.1) -- (0,1.1)-- (0,0.8);
 \draw [thick] (5-10,0) -- (7-10,0) -- (6-10, 1.73) -- (5-10,0);
 \draw [ultra thick] (5-10,0) -- (6-10, 1.73/3) -- (7-10,0) -- (6-10, 1.73/3) -- (6-10, 1.73);
\end{tikzpicture}
\caption{Conjectured shortest opaque set for the equilateral triangle with length $\sqrt{3}$ (left) and three sides of a rectangle (right).}
\end{figure}
\end{center}
\vspace{-20pt}

The constant $1/2$ in $L \geq |\partial \Omega|/2$ is optimal which can be seen by taking a $1 \times \varepsilon$ rectangle where $0 < \varepsilon \ll 1$. Taking three sides of the rectangle gives an opaque set with length $1+2\varepsilon$ while the Jones bound guarantees $L \geq 1+\varepsilon$. However, it should usually be far from optimal. Many of these questions have been extensively studied, we refer to \cite{asi, brakke, dumi0, dumi, dumi1, dumi2, faber, faber2, kawohl, kawohl2, kawohl3, maz, sen}.

\subsection{Main Result} In what follows, we consider opaque sets $\mathcal{O}$ that are given as a countable union of line segments. A limiting argument shows that this is not an actual restriction, see \cite[Lemma 1]{dumi1} or \cite[Lemma 4]{kawa}, but it simplifies exposition. We prove a stability version of Jones' inequality. If the Jones bound is nearly attained, then it will be shown that this tells us a lot about the structure of the corresponding opaque set $\mathcal{O}$: the orientation of the line segments in $\mathcal{O}$ has to be very similar to the orientation of the boundary $\partial \Omega$.

\begin{center}
    \begin{figure}[h!]
        \begin{tikzpicture}
            \draw [ultra thick,red] (0,0) -- (1.5,0);
           \draw [ultra thick, brown] (-1,0) -- (1,2);
           \draw [ultra thick, blue] (-1, 1) -- (-1,2);    
              \draw [ultra thick, blue] (1.2, 1) -- (1.2,1.5);    
           \draw [thick] (4,0) -- (8,0);
           \draw [thick] (4, -0.1) -- (4, 0.1);
         \draw [thick] (5, -0.1) -- (5, 0.1);
         \draw [thick] (6, -0.1) -- (6, 0.1);
            \draw [thick] (7, -0.1) -- (7, 0.1);  
            \draw [thick] (8, -0.1) -- (8, 0.1);
          \node at (4, -0.3) {0}; 
              \node at (5, -0.3) {$\pi/2$};
          \node at (6, -0.3) {$\pi$};
             \node at (7, -0.3) {$3\pi/2$};
          \node at (8, -0.3) {$2\pi$};
             \draw [ultra thick, red] (4,0) -- (4, 0.75); 
                      \draw [ultra thick, brown] (4.5,0) -- (4.5, 1.41); 
                 \draw [ultra thick, blue] (5,0) -- (5, 0.75);
          \draw [ultra thick, red] (6,0) -- (6, 0.75); 
                      \draw [ultra thick, brown] (6.5,0) -- (6.5, 1.41); 
          \draw [ultra thick, blue] (7,0) -- (7, 0.75); 
        \end{tikzpicture}
        \caption{A set of line segments $\mathcal{O}$ (colored by angle) and the corresponding `angular orientation' measure $\mu_{\mathcal{O}}$.}
        \label{fig:enter-label}
    \end{figure}
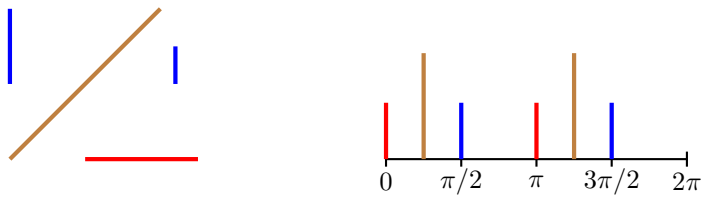
\end{center}

For each set of line segments $\mathcal{O} \subset \mathbb{R}^2$, we construct an associated measure $\mu_{\mathcal{O}}$ on $[0,2\pi) \cong \mathbb{T} \cong \mathbb{S}^1$ as follows: for any line segment $\ell_i$ of length $|\ell_i|$ making angle $0 \leq \alpha \leq 2\pi$ with the $x-$axis in the usual counterclockwise orientation, we add the two scaled Dirac measures $|\ell_i|/2 \cdot \delta_{\alpha}$ and $|\ell_i|/2 \cdot \delta_{\pi + \alpha}$ to the measure $\mu_{\mathcal{O}}$. This is to account for the fact that any line segment `equally points in two directions'. By construction, the total measure is exactly the length of all line segments
$ \mu_{\mathcal{O}}([0,2\pi)) = \mathcal{H}^1(\mathcal{O}).$
The measure $\mu_{\mathcal{O}}$ is uniquely determined by the opaque set $\mathcal{O}$. We construct a second measure $\mu_{\partial \Omega}$ by applying the same construction procedure to the boundary of the convex set $\partial \Omega$ and then scaling everything down by a factor of 2. If $\partial \Omega$ is piecewise linear, the procedure is identical; if not, then we perform a limit procedure which is easy to do since Rademacher's Theorem guarantees that the boundary of a convex set is piecewise differentiable almost everywhere. Because of the rescaling by a factor of 2 (whose purpose will become clear later, in Lemma 2), we have
$ \mu_{\partial \Omega}([0,2\pi)) = |\partial \Omega|/2.$
 For example, if $\Omega$ is the unit disk, then $\mu_{\partial \Omega} = dx/2$ is the rescaled Lebesgue measure. In the case of $\Omega = [0,1]^2$, we have
$$ \mu_{\partial\Omega} = \frac{\delta_0 + \delta_{\pi/2} + \delta_{\pi} + \delta_{3\pi/2}}{2}.$$
We can now state our main result which, informally speaking, says that if an opaque set $\mathcal{O}$ has length $L$ very close to $|\partial \Omega|/2$, then $\mu_{\mathcal{O}}$ and $\mu_{\partial \Omega}$ are very similar measures and their distance can be quantified in a suitable Sobolev space.

\begin{thm} Let $\Omega \subset \mathbb{R}^2$ be a bounded convex domain, let $\mathcal{O}$ be an opaque set of length $L$ and let $\mu_{\mathcal{O}}$ and $\mu_{\partial \Omega}$ be the associated measures. Then
$$ \left\| \mu_{\mathcal{O}} - \mu_{\partial \Omega} \right\|_{\dot H^{-2}(\mathbb{T})} \leq  \frac{ L^{1/4}}{\sqrt{2}} \cdot \left(L - \frac{|\partial \Omega|)}{2}\right)^{3/4}.$$
\end{thm}
\textit{Remarks}.
\begin{enumerate}
    \item For any measure $\nu$ the homogeneous Sobolev space $\dot H^{-2}(\mathbb{T})$ has a number of equivalent definitions; we work with one based on the Fourier transform
    $$ \| \nu \|_{\dot H^{-2}(\mathbb{T})}^2 := \sum_{\ell \in \mathbb{Z} \atop \ell \neq 0} \frac{ \left| \widehat{\nu}(\ell) \right|^2}{\ell^4}.$$
\item The inequality is true for all opaque sets; it is particularly informative when the Jones inequality is nearly satisfied, that is, when $L - |\partial \Omega|/2$ is small.
    \item Duality makes the result easier to interpret: for any smooth $2\pi-$periodic function $\phi:[0,2\pi] \rightarrow \mathbb{R}$ with mean value 0
    $$ \left| \int_0^{2\pi} \phi(\theta) d\mu_{\mathcal{O}}(\theta) - \int_0^{2\pi} \phi(\theta) d\mu_{\partial \Omega}(\theta) \right| \leq \|\phi\|_{\dot H^2} \cdot \left\| \mu_{\mathcal{O}} - \mu_{\partial \Omega} \right\|_{\dot H^{-2}}.$$
    By taking suitable test functions $\phi$, we see that proximity of the measures in $\dot H^{-2}(\mathbb{T})$ implies that they assign similar values to smooth functions. 
    \item For explicit special cases, easier proofs of variations of this result expressing the same principle may exist. We give an example for $[0,1]^2$ in \S 1.3.
\end{enumerate}

\subsection{A special case.}
 Returning to the unit square $[0,1]^2$, the main result shows that if $\mathcal{O}$ is a (hypothetical) opaque set with length is close to 2, then $\mathcal{O}$ has to be mostly comprised of line segments that are either very nearly parallel to the $x-$axis or very nearly parallel to the $y-$axis. There is a direct proof of such a result.

\begin{proposition}
Let $\Omega = [0,1]^2$ and let $\mathcal{O}$ be an opaque set. Then, for all $\eta > 0$,
\begin{enumerate}
    \item either its length is large
    $$ L \geq 2 + \eta \qquad$$
    \item or, for any $0 \leq \beta \leq \pi/4$, if $J_{\beta} \subset [0,2\pi)$ denotes the angles that are $\geq \beta$ from both sides of the $x-$axis and $y-$axis, then
    $$ \mu_{\mathcal{O}}(J_{\beta}) \leq  \frac{\eta}{1 - \cos{(\beta)}}.$$
 \end{enumerate}
\end{proposition}

This imposes some additional structure on hypothetical opaque sets that are fully contained in $[0,1]^2$ with a length close to 2. It is conjectured (see Fig. 1) that the shortest opaque set is fully contained in $[0,1]^2$. Using the Proposition we could consider lines with an angle of $\pm 45^{\circ}$ sweeping out the four disjoint corners of the unit square, see Fig. 4 (left). For each corner, one sweeps out a length of $1/\sqrt{8}$. Each of the lines has to hit the opaque set. If the opaque set has length close to 2, then most of it is either nearly parallel to the $x-$axis or nearly parallel to the $y-$axis and the length of the opaque set in each corner has to be nearly $\sqrt{2}/ \sqrt{8} = 1/2$. The four corners by themselves already have to contain very nearly length $2$. 

\begin{center}
    \begin{figure}[h!]
        \begin{tikzpicture}[scale=1.2]
        \draw (-5, 0) -- (-2,0) -- (-2,3) -- (-5, 3) -- (-5, 0);
        \draw [<->] (-2.1, 2.9) -- (-2.7, 2.3);
        \node at (-2.7, 2.8) {\small $8^{-1/2}$};
        \draw [dashed] (-5.5, 1.1) -- (-3, 3.6);
           \draw [dashed] (-5.5, 1.2) -- (-3.1, 3.6); 
                  \draw [dashed] (-5.5, 1.3) -- (-3.2, 3.6); 
        \draw [dashed] (-5.5, 2) -- (-3, -0.5);
          \draw [dashed] (-5.5, 1) -- (-3, 3.5);
           \draw [dashed] (-1.5, 2) -- (-4, -0.5);
         \draw [dashed] (-1.5, 1) -- (-4, 3.5);
        \end{tikzpicture}
        \caption{Four corners with the top left being `swept out'.}
        \label{fig:enter-label}
    \end{figure}
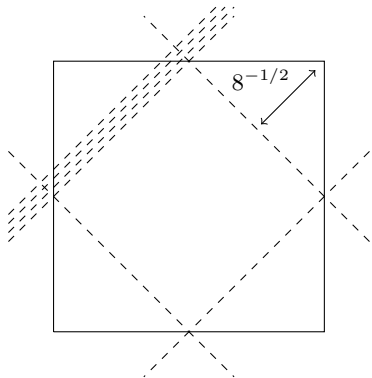
\end{center}

This appears to be a helpful additional piece of information, it severely restricts what a short opaque set can do. Our result generalizes this argument whenever we are close to the Jones bound. A similar idea also appeared in work of Izumi \cite{iz} for the equilateral triangle: Izumi arrived at a similar conclusion by means of a linear program and then used this to break the Jones barrier. One consequence of our main result is that this type of argument is now possible for any shape. In particular, it suggests that opaque sets for the unit disk with length close to $\pi$ would have to be very nearly `uniformly distributed in direction'.

\subsection{Outlook} Most of the existing arguments make heavy use of the angular distribution of the opaque set while putting less emphasis on their location in physical space. Presumably one should think of opaque sets as one-dimensional sets in $\mathbb{R}^2 \times \mathbb{S}^1$ but it is maybe less clear how this could be implemented. There is another way to think about the Jones bound. Let $X$ be a random variable defined as follows, we pick a `randomly' chosen line (`random' with respect to the kinematic measure placed on lines intersecting $\Omega$ but this is a technical detail) and count the number of intersections that the line has with $\mathcal{O}$. Crofton's formula tells us that $\mathbb{E}X$ is directional proportional to the length of $\mathcal{O}$ and $X \geq 1$ implies $\mathbb{E}X \geq 1$ which, after some computation, implies the Jones bound. It is clear that one should have $\mathbb{E}X \geq 1 +\varepsilon(\Omega)$ but it is less clear how to prove this. It was recently shown \cite{stein} that $\mathbb{E} X^2$ can be expressed in terms of energy functionals of the form
$$ E(\mathcal{O}) = \int_{\mathcal{O}} \int_{\mathcal{O}}\frac{\left|\left\langle n(x), y - x \right\rangle  \left\langle  y - x, n(y) \right\rangle   \right| }{\|x - y\|^{3}}~d \mathcal{H}^1(x) d\mathcal{H}^1(y),$$
where $n(x)$ and $n(y)$ are the normal vectors in $x,y \in \mathcal{O}$. Combined with our results above, it seems that this could be used to prove bounds along the lines of  $\mathbb{E}X^2 \geq 1 +\varepsilon_2(\Omega)$. However, these by themselves do not imply $\mathbb{E}X \geq 1 +\varepsilon(\Omega)$.

\section{Proofs}
\subsection{Jones' inequality.} We start by presenting the standard proof of Jones' inequality; along the way, we establish some of the notation that is used later. Let $\Omega \subset \mathbb{R}^2$ be a convex set and let $f(\theta)$ be the length of its projection onto the $x-$axis after having been rotated by angle $\theta$. The main inequality is
\begin{align} \label{ineq:proj}
\forall ~0\leq \theta \leq 2\pi \qquad \quad \int_0^{2\pi} |\cos{(\theta-\alpha)}| d\mu_{\mathcal{O}}(\alpha) \geq f(\theta).
 \end{align}
 It follows from assuming that all line segments cast a different shadow towards direction $\theta$, no overlap, and then compute the length of all these shadows: if it were shorter than the shadow cast by $\Omega$, then there would be a line intersecting $\Omega$ but not intersecting $\mathcal{O}$.
  Cauchy's surface area formula \cite{cauchy} implies that the mean-width of $\Omega$, the average value of $f(\theta)$, is directly proportional to the length of its boundary  
$$ \int_0^{2\pi} f(\theta) = 2 \cdot |\partial \Omega|.$$
Integrating over $\theta$ implies the Jones bound
\begin{align*}
    4 \cdot L &= \left( \int_0^{2\pi} |\cos{(\theta)}|d\theta\right) \left( \int_0^{2\pi} 1~ d\mu_{\mathcal{O}}(\alpha)\right) \\
    &=   \int_0^{2\pi} \left[ \int_0^{2\pi} |\cos{(\theta-\alpha)}| d\mu_{\mathcal{O}}(\alpha)\right]  d \theta \geq \int_0^{2\pi} f(\theta) d\theta = 2 \cdot |\partial \Omega|.
\end{align*}

\subsection{Refining the argument.} We introduce the abbreviation
$$ g(\theta) =  \int_0^{2\pi} |\cos{(\theta-\alpha)}| d\mu_{\mathcal{O}}(\alpha).$$
$g(\theta)$ only depends on the opaque set $\mathcal{O}$, $f(\theta)$ only depends on $\Omega$.
The main inequality is $g(\theta) \geq f(\theta)$ and it is an easy observation that this is either close to being attained for most $\theta$ or there is some room for improvement.
\begin{lemma}
    The length of the opaque set satisfies
    $$ L \geq \frac{|\partial \Omega|}{2} + \frac{1}{4}  \int_{0}^{2\pi}(g(\theta)-f(\theta)) d \theta.$$
\end{lemma}
\begin{proof}
    This follows from the previous argument since
    \begin{align*}
        4 \cdot L &= \int_{0}^{2\pi} g(\theta) d \theta = \int_{0}^{2\pi} f(\theta) + (g(\theta)-f(\theta)) d \theta\\
        &= 2 \cdot |\partial \Omega| +  \int_{0}^{2\pi}(g(\theta)-f(\theta)) d \theta.
    \end{align*}
\end{proof}

This shows that $g(\theta) - f(\theta)$ being large in $L^1$ leads to an improved lower bound for $L$. We prove an analogous result for the $L^2-$norm (Lemma 5). This requires a reverse H\"older inequality which we achieve by combining control of the Lipschitz constant (Lemma 3) with a bound on the maximum (Lemma 4). We also note an integral representation for the projection width $f(\theta)$ using $\mu_{\partial \Omega}$.

\begin{lemma} We have
     $$ f(\theta) =  \int_0^{2\pi} |\cos{(\theta-\alpha)}| d\mu_{\partial \Omega}(\alpha).$$
\end{lemma}
\begin{proof}
    Think of $\partial \Omega$ as a union of little line segments (or, more formally, approximate $\partial \Omega$ by these). The integral then computes the total length of the shadow in direction $\theta$ \textit{weighted by overlap}. However, $\Omega$ is convex, and thus the overlap is always a factor of 2, (Lebesgue-)almost every line passes through $\partial \Omega$ exactly twice. This is compensated by the factor $1/2$ in the definition of $\mu_{\partial \Omega}$.
\end{proof}

\begin{lemma}
$g(\theta)$ and   $f(\theta)$ are Lipschitz-continuous with Lipschitz constant $L$, the length of the opaque set.
\end{lemma}
\begin{proof}
    The statement for $g(\theta)$ follows at once from the formula
    $$ g(\theta) =  \int_0^{2\pi} |\cos{(\theta-\alpha)}| d\mu_{\mathcal{O}}(\alpha),$$
 the fact that $\left|\cos{(\theta)}\right|$ is Lipschitz continuous with Lipschitz constant 1 and that the measure has total mass $\mu_{\mathcal{O}}([0,2\pi)) = L$. The second statement is similar: we use Lemma 2 and the same argument to deduce that $f$ is Lipschitz constinuous with Lipschitz constant $|\partial \Omega|/2$ which, by Jones' bound, is smaller than $L$,
\end{proof}

\begin{lemma} We have 
$$ \int_{0}^{2\pi} (g(\theta) - f(\theta))^2~ d\theta \leq  8 \cdot \sqrt{L} \cdot \left(L - \frac{|\partial \Omega|)}{2}\right)^{3/2}.$$
\end{lemma}
\begin{proof}
    The function $h(\theta)= g(\theta) - f(\theta)$ is the difference of two $L-$Lipschitz functions and thus at most $2L-$Lipschitz. If $h$ assumes its maximum in $h(\theta^*) = M$, then
    $$ h(\theta) \geq \max\left\{ M - 2L |\theta - \theta^*|, 0 \right\}.$$
    Integrating over $\theta$ gives
    $$ \int_0^{2\pi} (g(\theta) - f(\theta)) ~d\theta \geq \frac{M^2}{L}.$$
    Combining this with Lemma 1
    \begin{align*}
          L \geq \frac{|\partial \Omega|}{2} + \frac{1}{4}  \int_{0}^{2\pi}(g(\theta)-f(\theta)) d \theta \geq \frac{|\partial \Omega|}{2} + \frac{M^2}{4L}.
    \end{align*}
This can also be rephrased as
$$ M =  \max_{0\leq \theta \leq 2\pi} (g(\theta) - f(\theta)) \leq \sqrt{4 L^2 - 2 L \cdot |\partial \Omega|}.$$
At this point, we argue
\begin{align*}
     \int_{0}^{2\pi} (g(\theta) - f(\theta))^2 d\theta &\leq \left(\max_{0\leq \theta \leq 2\pi} (g(\theta) - f(\theta)) \right) \int_{0}^{2\pi} (g(\theta) - f(\theta)) d\theta\\
     &\leq  \sqrt{4 L^2 - 2 L \cdot |\partial \Omega|}  \int_{0}^{2\pi} (g(\theta) - f(\theta)) d\theta\\
     &= \sqrt{4 L^2 - 2 L \cdot |\partial \Omega|} \cdot (4L - 2 |\partial \Omega|) \\
     &= 8 \cdot \sqrt{L} \cdot \left(L - \frac{|\partial \Omega|)}{2}\right)^{3/2}.
\end{align*}
\end{proof}

\subsection{Proof of the Theorem}
\begin{proof} The idea is relatively simple: write $f(\theta)$ and $g(\theta)$ as convolutions of $\left|\cos{(\theta)}\right|$ with $\mu_{\partial \Omega}$ and $\mu_{\mathcal{O}}$, respectively, take a Fourier transform and show that $\|f-g\|_{L^2(\mathbb{T})}$ has an alternative representation which, up to a constant, is comparable to the Sobolev norm. The non-degeneracy of the Fourier coefficients of $\left|\cos{(\theta)}\right|$, that they admit a bound from below for even frequencies, plays a crucial rule. The functions $1/(2\pi)$ and $e^{i \ell x}/\sqrt{2 \pi}$ for $\ell \neq 0$ form an orthonormal basis, we expand the measure $\mu_{\mathcal{O}}$ into Fourier series
    $$    \mu_{\mathcal{O}} = \frac{L}{2\pi} + \sum_{\ell \neq 0}  \left( \int_0^{2\pi} \frac{e^{-i \ell \theta}}{\sqrt{2\pi}} d\mu_{\mathcal{O}}(\theta)  \right) \frac{e^{i \ell \theta}}{\sqrt{2\pi}}$$
    and abbreviate this by defining the Fourier coefficients as
    $$    \mu_{\mathcal{O}} = \frac{L}{2\pi} +     \frac{1}{2\pi} \sum_{\ell \neq 0}  \widehat{\mu_{\mathcal{O}}}(\ell) e^{i \ell \theta}  \qquad \mbox{where} \qquad  \widehat{\mu_{\mathcal{O}}}(\ell) =\int_0^{2\pi} e^{-i \ell \theta} d\mu_{\mathcal{O}}(\theta).$$
Note that the measure $\mu_{\mathcal{O}}$ is singular and these equations are only valid in the weak sense and not in $L^p(\mathbb{T})$. Since $\mu_{\mathcal{O}}$ is, by construction, invariant under the shift $\theta \rightarrow \theta + \pi$, the Fourier coefficients associated to all odd frequencies $\ell $ vanish and only even frequencies $\ell \in 2 \mathbb{Z}$ remain. The same is true for $\mu_{\partial \Omega}$ which admits the same type of expansion. The function $\left|\cos{\theta}\right|$ can be written as
       $$ |\cos{(\theta)}| = \frac{4}{2\pi} +  \frac{1}{2\pi}\sum_{\ell \in \mathbb{Z} \atop \ell \neq 0}^{} a_{\ell} e^{i \ell \theta}$$
       where a short computation shows that, for $\ell \neq 0$,
       $$ a_{\ell} = \frac{4}{\ell^2 -1} \cdot \begin{cases}
           0 \qquad &\mbox{if}~\ell~\mbox{is odd} \\
         -1  &\mbox{if}~\ell \equiv 0 ~(\mbox{mod} ~4) \\
      1  &\mbox{if}~\ell \equiv 2 ~(\mbox{mod} ~4).
       \end{cases}$$
This allows us to compute the Fourier series of $g(\theta)$ explicitly as
\begin{align*}
 g(\theta) &= \int_{0}^{2\pi} |\cos{(\theta - \alpha)}| d\mu_{\mathcal{O}}(\alpha) \\
 &= \int_{0}^{2\pi} \left(  \frac{4}{2\pi} +  \frac{1}{2\pi}\sum_{\ell \in \mathbb{Z} \atop \ell \neq 0}^{} a_{\ell} e^{i \ell (\theta-\alpha)} \right) \left(  \frac{L}{2\pi} +     \frac{1}{2\pi} \sum_{m \neq 0}  \widehat{\mu_{\mathcal{O}}}(m) e^{i m \alpha}  \right)d\alpha \\
 &= \frac{4L}{2\pi} + \sum_{\ell \in \mathbb{Z} \atop \ell \neq 0}  \int_{0}^{2\pi}  \frac{a_{\ell}}{2\pi}  e^{i \ell (\theta-\alpha)} \frac{\widehat{\mu_{\mathcal{O}}}(\ell)}{2\pi}  e^{i \ell \alpha} d\alpha =  \frac{4L}{2\pi} +  \sum_{\ell \in \mathbb{Z} \atop \ell \neq 0} \frac{a_{\ell}}{2\pi}   \cdot \widehat{\mu_{\mathcal{O}}}(\ell)  \cdot  e^{i \ell \theta}.
\end{align*}
Using Lemma 2, the exact same argument also shows
$$ f(\theta) = \frac{2 |\partial \Omega|}{2\pi} +  \sum_{\ell \in \mathbb{Z} \atop \ell \neq 0} \frac{a_{\ell} }{2\pi} \cdot \widehat{\mu_{\partial \Omega}}(\ell)  \cdot  e^{i \ell \theta}.$$
Therefore
$$ g(\theta) - f(\theta) = \frac{4L - 2|\partial \Omega|}{2\pi} + \sum_{\ell \in \mathbb{Z} \atop \ell \neq 0}  \frac{a_{\ell}}{2\pi}  \cdot \left( \widehat{\mu_{\mathcal{O}}}(\ell) - \widehat{\mu_{\partial \Omega}}(\ell) \right) \cdot  e^{i \ell \theta} $$
from which we deduce that
$$ \int_{0}^{2\pi} (g(\theta) - f(\theta))^2 d\theta = \frac{(4L - 2|\partial \Omega|)^2}{2\pi} +  \sum_{\ell \in \mathbb{Z} \atop \ell \neq 0}  \frac{a_{\ell}^2}{2\pi}  \left| \widehat{\mu_{\mathcal{O}}}(\ell) - \widehat{\mu_{\partial \Omega}}(\ell) \right|^2.$$
Since the measures are only supported on even Fourier frequencies and, for even frequencies $\ell$, one has $|a_{\ell}| \geq 4/\ell^2$, one deduces
\begin{align*}
     \int_{0}^{2\pi} (g(\theta) - f(\theta))^2 d\theta &\geq  \sum_{\ell \in \mathbb{Z} \atop \ell \neq 0}  \frac{a_{\ell}^2}{2\pi}  \left| \widehat{\mu_{\mathcal{O}}}(\ell) - \widehat{\mu_{\partial \Omega}}(\ell) \right|^2 \\
     &\geq \sum_{\ell \in \mathbb{Z} \atop \ell \neq 0}  \frac{1}{2\pi} \frac{16}{\ell^4} \left| \widehat{\mu_{\mathcal{O}}}(\ell) - \widehat{\mu_{\partial \Omega}}(\ell) \right|^2 \\
     &= 16 \cdot \left\| \mu_{\mathcal{O}} - \mu_{\partial \Omega} \right\|_{\dot H^{-2}}^2,
\end{align*}
where the factor of $2\pi$ was absorbed into the definition of the Fourier coefficient (to make it, as opposed to above, an inner product with respect to an orthonormal element). Finally, with Lemma 4,
\begin{align*}
    16 \cdot \left\| \mu_{\mathcal{O}} - \mu_{\partial \Omega} \right\|_{\dot H^{-2}}^2 \leq   \int_{0}^{2\pi} (g(\theta) - f(\theta))^2 d\theta \leq  8 \cdot \sqrt{L} \cdot \left(L - \frac{|\partial \Omega|)}{2}\right)^{3/2}.
\end{align*}
\end{proof}

\subsection{Proof of the Proposition}
 
\begin{proof} As above, we use the abbreviation
$$ g(\theta) =  \int_0^{2\pi} |\cos{(\theta-\alpha)}| d\mu_{\mathcal{O}}(\alpha)$$
for the maximal length of the shadow at angle $\theta$. Since $[0,1]^2$ is completely explicit,
we can compute the length of the shadow $f(\theta)$ in closed form: either by direct computation
or via Lemma 2, we arrive at 
\begin{align*}
     f(\theta) &= \frac{1}{2} \left| \cos{(\theta)} \right| +   \frac{1}{2} \left| \cos{\left(\theta- \frac{\pi}{2}\right)} \right| +  \frac{1}{2} \left| \cos{\left(\theta - \pi\right)} \right| +  \frac{1}{2} \left| \cos{\left(\theta- \frac{3\pi}{2}\right)} \right| \\
&= \left|\cos{(\theta)}\right| + \left| \sin{(\theta)} \right|.
\end{align*}
In the regime of interest, there is a further simplification
\begin{align} \label{eq:simply}
     \forall~0 \leq \theta \leq \frac{\pi}{2} \qquad f(\theta) = \sqrt{2} \left| \cos\left( t - \frac{\pi}{4} \right) \right|.
\end{align}
 Using inequality \ref{ineq:proj} simultaneously in two directions
\begin{align*}
 g(\pi/4) + g(3\pi/4) &=   \int_0^{2\pi} \left( |\cos{(\pi/4-\alpha)}| +   |\cos{(3\pi/4-\alpha)}| \right) d\mu_{\mathcal{O}}(\alpha) \\
 &\geq f(\pi/4) + f(3\pi/4) = 2 \sqrt{2}.
\end{align*}   
Now we exploit a symmetry: as it turns out 
$$ |\cos{(\pi/4-\alpha)}| +   |\cos{(3\pi/4-\alpha)}| = f(\alpha + \pi/4).$$
We define the set of angles that make at least angle $0 \leq \beta \leq \pi/4$ with the $x-$axis and the $y-$axis, this set is
$$ J_{\beta} = \left[\beta, \frac{\pi}{2} - \beta\right] \cup  \left[\frac{\pi}{2} + \beta, \pi - \beta\right] \cup 
\left[\pi + \beta, \frac{3\pi}{2} - \beta\right] \cup \left[\frac{3\pi}{2} + \beta, 2\pi - \beta\right].$$
Suppose that total length $A$ of the opaque set has an angle close to $x-$axis or $y-$axis and let the length of the complement be denoted by $B$, 
$$ A = \mu_{\mathcal{O}}\left( J_{\beta}\right) \qquad \mbox{and} \qquad B = \mu_{\mathcal{O}}\left( [0,2\pi) \setminus J_{\beta}\right).$$
A simple computation shows
$$ \max_{t } f(t+\pi/4) = \sqrt{2} \qquad \mbox{and} \qquad \max_{t \in J_{\beta}} f(t+\pi/4) = f(\beta + \pi/4).$$
Then
\begin{align*}
    2 \sqrt{2} &\leq   \int_0^{2\pi} \left( |\cos{(\pi/4-\alpha)}| +   |\cos{(3\pi/4-\alpha)}| \right) d\mu_{\mathcal{O}}(\alpha)\\ 
    &\leq \sqrt{2} A + f(\beta + \pi/4) B.
\end{align*}
Now we argue that the set is either large $A+B \leq 2 + \eta$ (in which case we are done), or it is small and $A \leq 2 + \eta  -B$ and we deduce
$$   2 \sqrt{2} \leq \sqrt{2} \left(  2 + \eta  -B \right) + f(\beta + \pi/4) B$$
which can be rewritten as
$$ B \leq \frac{\sqrt{2} }{\sqrt{2} - f(\beta+\pi/4)} \eta.$$
Recalling now that $0 \leq \beta \leq \pi/4$, we may use \eqref{eq:simply} to deduce the desired result.
\end{proof}


\begin{thebibliography}{100}
\bibitem{stein} S. Steinerberger, Quadratic Crofton and sets that see themselves as little as possible, Monatshefte für Mathematik, 1--13.
\end{thebibliography}
\end{document}